\documentclass[11pt,fleqn,english]{article}
\usepackage[utf8]{inputenc}
\usepackage{csquotes}

\usepackage[backend=biber,style=numeric,sortcites,natbib=true,sorting=none,hyperref=true,backref=true]{biblatex}
\usepackage[a4paper]{geometry}
\usepackage[english]{babel}

\usepackage{mathtools}
\usepackage{amssymb,amsfonts}
\usepackage{amsthm}
\usepackage{authblk}

\usepackage[all]{xy}
\usepackage{float}
\usepackage{subfig}
\usepackage{graphicx}
\usepackage[font=small,labelfont=bf]{caption}
\usepackage{hyperref}

\usepackage{dsfont}
\usepackage{mathrsfs}

\newtheorem*{maintheorem}{Main Theorem}
\newtheorem*{maincorollary}{Main Corollary}
\newtheorem{proposition}{Proposition}[section]
\newtheorem{lemma}{Lemma}[section]

\newtheorem{definition}{Definition}[section]
\newtheorem{remark}{Remark}[section]
\newtheorem{question}{Question}
\newtheorem{theorem}{Theorem}[section]

\newcommand{\Real}{\mathbb R}       
\newcommand{\Nat}{\mathbb N}        
\newcommand{\oo}{\infty}            
\newcommand{\acts}{\curvearrowright}

\newcommand{\Lcal}{\mathcal{L}}						
\newcommand{\mm}{\mathrm{m}}

\newcommand{\PM}[1][M]{\ensuremath{\mathscr{Pr}(#1)}}                       

\newcommand{\PTM}[2]{\ensuremath{\mathscr{Pr}_{\cramped[\scriptstyle]{#1}}(#2)}}      	

\newcommand{\one}{\mathds{1}}                   							 

\newcommand{\entPn}[1][P]{\ensuremath{H_m((P)_0^{-n+1})}}

\newcommand{\CM}[1][M]{\ensuremath{\mathcal{C}(#1)}}

\DeclarePairedDelimiterXPP\clo[1]{\mathtt{cl}}(){}{#1}				

\newcommand{\der}{\operatorname{d}\hspace{-2pt}}

\newcommand{\diam}{\operatorname{diam}}

\newcommand{\F}{\ensuremath{\mathcal{F}}}

\newcommand{\Fs}{\ensuremath{\mathcal{W}^s}}

\newcommand{\Fu}{\ensuremath{\mathcal{W}^u}}

\newcommand{\Fcs}{\ensuremath{\mathcal{W}^{cs}}}
\newcommand{\Fcu}{\ensuremath{\mathcal{W}^{cu}}}

\newcommand{\Wu}[1]{\ensuremath{W^u(#1)}}

\newcommand{\Wcs}[1]{\ensuremath{W^{cs}(#1)}}

\newcommand{\Wsloc}[1]{\ensuremath{W^{s}_{\scriptstyle \mathrm{loc}}(#1)}}
\newcommand{\Wuloc}[1]{\ensuremath{W^{u}_{\scriptstyle \mathrm{loc}}(#1)}}

\newcommand{\Ptop}{P_{\scriptstyle \mathit{top}}}			

\newcommand{\Mat}[2]{\mathtt{Mat}_{#1}(#2)}

\newcommand{\Sl}[2]{\mathtt{Sl}_{#1}(#2)}

\newcommand{\hol}[1][\F]{\mathrm{hol}^{\scriptstyle #1}}

\newcommand{\ep}{\epsilon}

\newcommand{\muf}[1][\varphi]{\mathrm{m}_{\scriptstyle #1}}

\newcommand{\nux}[1][x]{\nu^u_{#1}}
\newcommand{\nuux}{\nu^u_{\underline{x}}}

\newcommand{\Jac}{\mathrm{Jac}}

\newcommand{\Coc}[1][\mathtt{f}]{\mathscr{Coc}(#1)}
\newcommand{\Eq}{\mathscr{Eq}}

\newcommand{\h}{\mathfrak{h}}
\newcommand{\g}{\mathfrak{g}}

\DeclareFontFamily{U}{mathb}{\hyphenchar\font45}
\DeclareFontShape{U}{mathb}{m}{n}{
      <5> <6> <7> <8> <9> <10> gen * mathb
      <10.95> mathb10 <12> <14.4> <17.28> <20.74> <24.88> mathb12
      }{}
\DeclareSymbolFont{mathb}{U}{mathb}{m}{n}
\DeclareMathSymbol{\toit}{3}{mathb}{'375}

\hypersetup{
	unicode=true,          
	pdftoolbar=true,        
	pdfmenubar=true,        
	pdffitwindow=false,     
	pdfstartview={FitH},   
	pdftitle={Rigidity of equilibrium states and unique quasi-ergodicity for horocyclic foliations},    
	pdfauthor={Pablo D. Carrasco and Federico Rodriguez-Hertz},     
	pdfkeywords={Equilibrium States, Conditional Measures, Unique Ergodicity}, 
	pdfnewwindow=true,      
	colorlinks=true,       
	linkcolor=blue,          
	citecolor=green,       
	filecolor=blue,     
	urlcolor=blue           
}
\usepackage[capitalise]{cleveref}

\makeatletter
\newcommand{\subjclass}[2][2020]{%
  \let\@oldtitle\@title%
  \gdef\@title{\@oldtitle\footnotetext{#1 \emph{MSC2020 classification:} #2}}%
}
\newcommand{\keywords}[1]{%
  \let\@@oldtitle\@title%
  \gdef\@title{\@@oldtitle\footnotetext{\emph{Keywords:} #1.}}%
}
\makeatother

\begin{document}

	\title{Rigidity of equilibrium states and unique quasi-ergodicity for horocyclic foliations}
	\author[1]{Pablo D. Carrasco \thanks{Partially supported by FAPEMIG Universal APQ-02160-21 and CNPq-Projeto Universal  404943/2023-3}}
	\author[2]{Federico Rodriguez-Hertz \thanks{Partially supported by Katok Chair in Mathematics}}
    \affil[1]{\emph{pdcarrasco@ufmg.br}, ICEx-UFMG, Avda. Presidente Antonio Carlos 6627, Belo Horizonte-MG, BR31270-90}
	\affil[2]{\emph{hertz@math.psu.edu}, Penn State, 227 McAllister Building, University Park, State College, PA16802}

  \date{\today}
  
  \maketitle

\selectlanguage{english}
\begin{abstract}
In this paper we prove that for topologically mixing metric Anosov flows their equilibrium states corresponding to H\"older potentials satisfy a strong rigidity property: they are determined only by their disintegrations on (strong) stable or unstable leaves. 

As a consequence we deduce: the corresponding horocyclic foliations of such systems are uniquely quasi-ergodic, provided that the corresponding Jacobian is H\"older, without any restriction on the dimension of the invariant distributions. This gives another proof of a result of Babillott-Ledrappier.
\end{abstract}

\section{Introduction and Main Results}

\subsection{Equilibrium states for metric Anosov flows}

Let $\mathtt f=(f_t)_t:M\to M$ be an Anosov flow with invariant bundle decomposition $TM=E^{s}\oplus E^c\oplus E^{u}$. As it is well known, both $E^s, E^u$ are integrable to leafwise smooth foliations $\Fs, \Fu$ (which in this article we refer as the $s-, u-$ horocyclic foliations) that are contracted and expanded respectively under $f_t, t>0$. More generally, in this work we will be concerned with metric Anosov flows (also called Smale flows), which were introduced by Pollicot (inspired in  the corresponding concept for homeomorphisms due to Ruelle) and gained recent interest due to their applicability to geometrical problems. Indeed, for a word hyperbolic group $G$, Gromov constructs a compact metric space $M$ which comes together with an action $G\acts M$ by isometries, and a transitive metric Anosov flow $\mathtt f:M\to M$, in such a way that $M/x\sim f_t(x)$ is $G$-equivariantly identified with $\partial^2 G=\{(g,h)\in \partial G:g\neq h\}$. The orbits of $\mathtt f$ are a model for the (coarse) behavior of geodesics in $G$. See \cite{GromovHypGr}. More recently, in the same setting \cite{Bridgeman2015} Bridgeman et al. construct for a (projective) Anosov representation $\rho:G\to \Sl{d}{\Real}$ a metric Anosov flow (which turns out to be a H\"older reparametrization of the previous flow $\mathtt f:M\to M$), and use it among other things to obtain information about the spectrum of the representation. We remark that in these contexts the dimension of the horocyclic laminations is typically larger than one.

In what follows $\mathtt f$ denotes a transitive metric Anosov flow with invariant laminations $\Fs,\Fu$; the definition is recalled in \Cref{sub:metricAnosov}, but if the reader is not acquainted by the theory for this introduction she/he can safely assume that $\mathtt f$ is a differentiable Anosov flow. Transitive metric Anosov flows come in two flavors: either the horocyclic laminations are minimal (each one of its leaves is dense), or none of its horocyclic leaves are dense. See \cite{AnosovFlow} (Plante's arguments also apply to metric Anosov flows). The first case corresponds to the flow $\mathtt f$ being topologically mixing, while the second is the so called suspension case (in particular, there is a global transversal to the flow).

A central objective here is showing that in the first case the horocyclic laminations are much more rigid that appear at first sight, and relate this with the ergodic properties of the underlying dynamics $\mathtt f$, particularly its equilibrium states corresponding to H\"older observables (see \Cref{sub:ruelle_perron_frobenius_operator} for a review of the theory). Connecting the properties of equilibrium states of hyperbolic systems with their induced measures along horocyclic leaves is a classical and important topic in (usually smooth) ergodic theory, which typically depends on the fact that the $s-, u-$ leaves carry a natural Lebesgue class induced by some Riemannian metric on the manifold. Given an invariant probability measure $\mm$ for the flow one can detect whether or not it is an equilibrium state for the geometric potential $\varphi(x)=\int_0^1 -\log|\det D_xf_t|_{E^u}|\der t$ by checking if the disintegrations of $\mm$ along $\Fu$ are in equivalent to the corresponding Lebesgue measures on leaves \cite{BowenRuelle}. The word disintegration above is used in the sense of Rohklin theory \cite{Rokhlin}: given a Borel probability measure $\mm$ and a countably generated partition $\xi=\{A_i\}_{i\in\Lambda}$ of $M$ with respect to $\mm$ ($i\neq j \Rightarrow \mm(A_i\cap A_j)=0$), there is an essentially unique family $\{\mm^\xi_x\}_{x\in M_0}$ of probability measures on $M$ (referred as the conditionals of $\mm$ with respect to $\xi$) satisfying 
\begin{enumerate}
	\item[S-1] $\mm(M_0)=1$,
	\item[S-2] $x,y\in A_i$ implies that $\mm_x^{\xi}=\mm_y^{\xi}$ assigns measure one to $A_i$, and
	\item[S-3] for a Borel set $A \subset M, \mm(A)=\int \mm^{\xi}_x(A)\der\mm(x)$.
\end{enumerate}

There is a subtle but important point is the discussion above: the partition by $u$-leaves is not measurable in the sense of Rohklin, so one has to consider measurable partitions subordinated to $\Fu$. Since these depend in general on $\mm$, to be able to compare we need our reference measures to be defined everywhere, and not almost everywhere with respect to some fixed particular measure. When considering the geometrical potential this is not an issue, but for general equilibrium states (even for the entropy maximizing, measure which corresponds to the zero potential) the existence of these reference measures is not clear at all.

This is precisely what we elucidate here.

\begin{maintheorem}\hypertarget{main}{}
Assume $\mathtt f=(f_t)_t:M\to M$ to be a topologically mixing metric Anosov flow, and let $\varphi:M\to\Real$ be a H\"older function. Then there exists a family $\{\nu^u_x\}_{x\in M}$ satisfying:

\begin{enumerate}
	\item $\nu^u_x$ is a Radon measure on the $\Wu{x}$.

	\item If $\mm$ is any probability, invariant or not, and $\xi=\{A_i\}_{i\in\Lambda}$ is a measurable partition adapted to $\Fu$ whose conditionals are of the form
	\[
	x\in A_i\Rightarrow \mm^{\xi}_x=\frac{\nu^u_x}{\nu^u_x(A)},
    \]
   then $\mm$ is the unique equilibrium state for the system $(\mathtt f,\varphi)$.  
\end{enumerate}
\end{maintheorem}

The adapted partitions referred above are measurable partitions with some convenient additional properties: they are measurable partitions whose atoms are subsets of $\Fu$ that contain relatively open neighborhoods of each point, and are increasing under the discretized action $\{f_n\}_{n\in \Nat}$ ($f_1(A_i)$ is union of atoms of $\xi$, for every $A_i\in \xi$). These are also universal in the sense that they are partitions with respect to every possible equilibrium state. Partitions of this were introduced to Sinai for uniformly hyperbolic system (see for example \cite{SinaiMarkov}), and generalized to non-uniformly expanding foliations by Ledrappier and Strelcyn \cite{LedStrelcyn}. We remark that these are constructed for invariant measures, but if $\mm$ does not admit such partitions, then it is a priori discarded as a possible equilibrium state.

\subsection{Uniqueness of the transversal measure and rigidity}

On first glance the main motif of the \hyperlink{main}{Main Theorem} is to construct the family of measures on the unstable lamination which are can be used to detect equilibrium states. While this may be true, more careful analysis shows that it reveals a very rigid structure of $\Fu$. Indeed, the previous theorem does not assume that the measure $\mm$ is invariant under the flow, which demonstrates that all information is encoded in how this measure disintegrates along $\Fu$. This is somewhat surprising: conditions S-2, S-3 for disintegrations show that $\mm$ can be uniquely determined by its conditionals with respect to $\xi$ together with the transverse measure, that is, the induced measure on $M/\xi$. Simple minded examples in product spaces show that these can be chosen fairly independently.

The reason for this behavior is that in fact, for a given (regular) Jacobian $J$ there is only one possible quasi-invariant transverse measure for the holonomy pseudo-group of $\Fu$ associated to $J$. Let us explain this.

Given a lamination $\F$ of a metric space $M$, one can use the lamination charts to make sense of what it means for a subset $D \subset M$ to be transverse to $\F$. In the case when $M$ is a manifold, $D \subset M$ is an embedded disc and $\F$ has continuous tangent bundle, this notion coincides with transversality in the differential topology sense between $D$ and every leaf of $\F$ that it intersects. We say that a family $\{\mu_x:x\in M\}$ is a \emph{transverse measure} if it satisfies:
\begin{enumerate}
	\item each $x\in M$ is contained in some transversal $D_x$, and
	\item $\mu_x$ is a non-trivial Borel measure on $D_x$.
\end{enumerate}

Given one of such measures and $y\in\F(x)$, one can use the holonomy of the foliation to compare $\mu_x$ and $\mu_{y}$, and least for nearby $x,y$: we denote $\hol_{y,x}: E_x \subset D_x\to E_y\subset D_y$, the holonomy transport, where $E_x, E_y$ are relatively open. 

\begin{definition}
The transverse measure $\{\mu_x\}_{x}$ of $\F$ is quasi-invariant if there exist a family of positive functions $\Jac_{\mu}=\{\Jac_{y,x}: E_x \to \Real_{>0}:y\in \F(x)\}$ and positive constants $\{C(y,x): y\in\F(x)\}$ such that for $y\in \F(x)$,
\[
	\hol_{x,y}\mu_{y}=C(y,x)\Jac_{x,y}\mu_{x}.
\] 
The family $\Jac_{\nu}$ is the Jacobian of the quasi-invariant measure. If $\Jac_{\nu}\equiv 1$ then the measure is said to be an invariant transverse measure. 
\end{definition}

In the setting that we are working it is convenient to use local center-stable sets as the family of transversals to $\Fu$, and we use implicitly this choice in what follows. Next we discuss the possible Jacobians, and remark that necessarily such a family is a $(\Real_{>0},\cdot)$-valued cocycle over the holonomy pseudo-group of $\Fu$. Here is a natural way to define such a family: let
\begin{equation}\label{eq:cocicles}
 \Coc=\{h:M\to\Real_{+}: h(x)=e^{\int_0^1 \varphi(f_t x)dt}\text{ for some H\"older function }\varphi\},
 \end{equation}
 and given $h\in\Coc$, define 
\begin{equation}\label{eq:hjacobiano}
\mathscr{h}=\{H_{x_0,y_0}:x_0,y_0\in M, y_0\in\Wu{x_0}\}
\end{equation}
with
\begin{equation}
H_{x_0,y_0}(x)=\prod_{j=1}^{\oo}\frac{h(f_{-j}\circ \hol[\Fu]_{y_0,x_0} x)}{h(f_{-j}(x))},\quad x\in\Wcs{x_0}, \hol[\Fu]_{y_0,x_0} x\in\Wcs{y_0}.
\end{equation}
It is then direct to verify that $\mathscr{h}$ is a multiplicative cocycle, and in fact one can check that for every equilibrium state of $\muf$ of $\mathtt f$ corresponding to a H\"older potential, the transverse measure of $\Fu$ is quasi-invariant, with a Jacobian of the previous form. This is discussed in our article \cite{EqStatesCenter} for differentiable Anosov flows, but the proof given extends without any change to the metric case. See \cite{ContributionsErgodictheory} for further discussion.

We are able to conclude the following. 

\begin{maincorollary}\hypertarget{maincoro}{}
In the hypotheses of the \hyperlink{main}{Main Theorem}, given $h\in\Coc$ there exists $\mu^{cs}=\{\mu_x^{cs}\}$ a transverse measure for $\Fu$ such that $\mu^{cs}$ is the unique quasi-invariant measure with Jacobian given by the family $\mathscr{h}$ determined by $h$. 	
\end{maincorollary}

The proof that the \hyperlink{main}{Main Theorem} implies the \hyperlink{maincoro}{Main Corollary} is exactly the same as the one given in \cite{ContributionsErgodictheory}, and thus we will not repeat it here. 

\subsection{Comparison with other literature}

The \hyperlink{main}{Main Theorem}, which improves the classical SRB theorem \cite{BowenRuelle}, appears in our recent paper \cite{ContributionsErgodictheory} under additional hypotheses: differentiability of the Anosov flow and one-dimensionality of the unstable distribution. Questions due to professor F. Labourie motivated us to reformulate a previous version without assuming any differentiable structure; as mentioned at the beginning of the introduction, this may be relevant for geometrical applications. We remark here that the family of measures $\{\nu_x\}_{x\in M}$ are not the same as the ones constructed by Haydn \cite{localproductstructureHaydn} (in particular, if $x,y$ are in the same unstable leaf then $\nu_x$ is typically different from $\nu_y$), although they are related. Uniqueness of the Haydn family is proven in \cite{geodesicbabillotledrappier}.

The \hyperlink{maincoro}{Main Corollary} was established originally by Babillot-Ledrappier \cite{geodesicbabillotledrappier}, and a different proof was given by Schapira \cite{SCHAPIRA2004} in the context of hyperbolic manifolds (in particular, algebraic). The method of the proof of the original result consist of establishing first the theorem for the symbolic model \cite{SymbHyp} of the flow, and then assembling everything back in the manifold. The second part is delicate, and Babillot-Ledrappier refer to the work \cite{Bowen1977} for this part, while the same statement is proven for the zero potential. On the other hand, the second proof relies heavily in the geometrical information that constant negative curvature provides. Our method is different and more analytical in flavor: given a H\"older potential on $M$ we push everything to the symbolic model $\Sigma$ of the dynamics under the flow of $u-$leaves and construct a family of averaging operators $\{R_n:\CM[\Sigma]\toit\}$. For each $n\geq 0$, $R_n$ is a modified version of the transfer operator associated to the equilibrium state $\muf$ of the potential restricted to natural $\sigma$-algebra of $\Sigma$ that consists of $n$-cylinders. Finally, we prove SOT convergence of these operators to the functional defined by $\muf$, and show that this can be transferred back to the manifold. The heart of the argument is that we have never used the transverse structure (the lamination $\Fcs$), and from this we deduce our results. This type of argument is more similar to Marcus' proof of the unique ergodicity of $\Fu$ for codimension-one Anosov flows \cite{Marcus1975}, which in turn generalizes the classical theorem of Furstenberg about equi-distribution of the horocycles in surfaces of constant negative curvature \cite{Furstenberg1973}. We remark however that Marcus' proof explicitly uses that the corresponding conditional measures along $\Fu$ behave quite tamely under iterations (the Jacobian is constant), which together with one-dimensionality of the leaves makes dealing wih the corresponding averaging operators more manageable than for general potentials. For the proof of the \hyperlink{maincoro}{Main Corollary} we do not make such assumptions, and thus we provide a complete direct proof of Babillot-Ledrappier result without further assumptions.

We finish this introduction by posing a general question. It follows by the results of this article that horocyclic foliations have strong rigidity properties in terms of their admitted quasi-invariant measures. On the other hand, the existence of the renormalizing dynamics (i.e., the associated hyperbolic system) is probably not necessary for this phenomena, as the result of Veech \cite{Veech1977} shows (see also \cite{Ratner1991}). We can thus ask:

\begin{question}
Let $\Gamma$ be a co-compact lattice in a lie group $G$, and let $U$ be a one-parameter unipotent subgroup of $G$. Consider $\alpha$ a differentiable multiplicative cocycle over $U\acts \Gamma/ G$. Does there exists a unique quasi-invariant measure for the action with Jacobian given by $\alpha$?.   
\end{question}

Veech answers positively the previous question when the cocycle is trivial.


\section{Prerequisites}\label{sec:prerequisites}

\subsection{Metric Anosov flows}\label{sub:metricAnosov}

 We follow the exposition of \cite{Bridgeman2015}. Let $M$ be a metric space. A collection of embeddings $\mathcal A=\{\phi_U=(\phi_U^1,\phi^2_U):U \subset M\to N_U\times K_U\}$, where $N_U, K_U$ are topological spaces, is a lamination atlas if
 \begin{itemize}
 	\item $\mathcal U=\{U\}$ is an open covering of $M$;
 	\item whenever $U\cap V\neq \emptyset$ it holds
 	\[
 	x, y\in U\cap V \text{ implies } \phi_U^1(x)=\phi_U^1(y)\Leftrightarrow \phi_U^2(x)=\phi_U^2(y). 
 	\]
 \end{itemize}
Level sets of $\phi_U^2$ (for $\phi_U\in \mathcal{A}$) are called plaques, and a plaque chain is a finite collection $P_1, \ldots, P_k$ of plaques such that $P_i\cap P_{i+1}\neq \emptyset, \forall i=0,\ldots,k-1$. Given such an atlas, it defines an equivalence relation $\F$ on $M$, where $x,y$ are equivalent if and only if they can be joined by a plaque chain. The partition induced by $\F$ (also denoted by $\F$) is referred as the lamination, whereas its atoms are called leaves. The collection of plaques induced by a given lamination atlas is called a plaquation. It is useful to consider as equivalent lamination atlases defining the same structure on $M$: the phrase ``lamination atlas of $\F$'' is used in this sense. Note that if $U \subset M$ is open then $\F$ defines naturally a lamination in $U$, which is denoted by $\F|U$.

Suppose now that $\F$ is a lamination and $\mathtt f=(f_t)_{t\in \Real}$ is a flow on $M$. The lamination $\F$ is invariant under $\mathtt f$ if each $f_t$ permutes the leaves of $\F$. An invariant lamination $\F$ is transverse to $\mathtt f$ if there exists an atlas $\mathcal B=\{\psi_U: N_U\times (K_U\times (-\epsilon_U,\epsilon_U))\}$ of $\F$ so that
\[
	x\in U, \psi(x)=(a,b,t_0)\Rightarrow \phi_t(x)=\psi_U^{-1}(a,b,t_0+t)\quad \forall |t+t_0|<\epsilon_U.
\]
In this case note that the same atlas $\mathcal B$ serves as a laminated atlas for the partition induced by the orbits of $\mathtt f$. Due to invariance, we can use this structure to extend the lamination $\F$ by saturating their leaves under $\mathtt f$ and obtain a new lamination $\F^c=\{f_{(-\oo,\oo)}(L):L\in \F\}$. By definition, each leaf of $\F^c$ is saturated by orbits of $\mathtt f$: if it is also saturated by leaves of $\F$ we say that $\F$ is complete under $\mathtt f$. 

Given $\F_1,\F_2$ laminations of $M$, we say that they are transvere if there exists a lamination atlas $\mathcal A=\{\phi_U=(\phi_U^1,\phi^2_U:U \subset M\to N_U\times K_U\}$ for $\F_1, \F_2$ such that each $\phi_U$ sends $\F_1|U$ to the horizontal lamination $\{N_U\times\{b\}:b\in K_U\}$, and $\F_2|$ to the vertical lamination $\{\{a\}\times K_U:a\in N_U\}$.

\begin{definition}
Let $(M,d)$ be a compact metric space and $\mathtt f=(f_t)_{t\in \Real}$ a continuous flow on it. We say that $\mathtt f$ is a metric Anosov flow if there exist laminations $\Fu,\Fs$ with some fixed associated plaquations $\mathcal P^u=\{P^u\}, \mathcal P^s=\{P^s\}$, and constants $C>0, 0<\lambda<1$ such that
\begin{enumerate}
	\item $\Fu,\Fs$ are invariant under $\mathtt f$ and complete.
	\item $\Fcs,\Fu$ and $\Fs, \Fcu$ are transverse.
	\item It holds for all $t\geq 0$,
	\begin{align*}
	x,y\in P^u\Rightarrow d(f_{-t}(x),f_{-t}(y))\leq C\lambda^{t}\\
	x,y\in P^s\Rightarrow d(f_{t}(x),f_{t}(y))\leq C\lambda^{t}.
	\end{align*}
\end{enumerate}
\end{definition}

For the rest of the article $\mathtt f=(f_t)_{t\in \Real}$ denotes a metric Anosov flow.

\subsection{Symbolic models for hyperbolic systems}

Given a matrix $A\in\Mat{d}{\{0,1\}}$ one considers the two-sided subshift of finite type that it determines,
\[
	\Sigma_A=\{\underline{x}=(x_n)_n\in\{0,1\}^{\mathbb Z}:A_{x_n x_{n+1}}=1,\forall n\in \mathbb Z\}.
\]
For $k, l\in\Nat$ and a word $a_{-k}\cdot a_{-k+1}\cdots a_{l}\in \{0,1\}^{l+k+1}$ with $A_{a_i,a_{i+1}}=1$ for all $i$, we denote
\[
	C(a_{-k}\cdots a_l)_{-k}=\{\underline{x}\in\Sigma_A:x_i=a_i, \forall -k\leq i\leq l\}.
\]
Note that 
\[
	C(a_{-k}\cdots a_l)_{-k}=\bigcap_{i=-k}^l\sigma^{-i}(C(a_i)_0).
\]
\begin{definition}
Sets of the previous form are called rectangles. If $k=l$ we say that the rectangle is symmetric.
\end{definition}

The topology in $\Sigma_A$ is metrizable, where a compatible metric is given as 
\[
	d(\underline x, \underline y)=\frac{1}{2^{N+1}};
\]
above $N$ is the size of the largest symmetric rectangle that contains $\underline{x}, \underline y$.

\smallskip 

Similar considerations can be applied to the one-sided shift spaces
\begin{align*}
\Sigma_A^{-}=\{\underline{x}=(x_n)_{n\leq 0}\in\{0,1\}^{-\Nat}:A_{x_n x_{n+1}}=1,\forall n<0\}\\
\Sigma_A^{+}=\{\underline{x}=(x_n)_{n>0}\in\{0,1\}^{\Nat^{\ast}}:A_{x_n x_{n+1}}=1,\forall n>0\}
\end{align*}
Note that any $\underline{x}\in\Sigma_A$ can be written uniquely in the form
\[
	\underline{x}=\underline{x}^-\cdot \underline{x}^+ \quad \underline{x}^-\in \Sigma_A^{-}, \underline{x}^+\in \Sigma_A^{+}
\]
where the $\cdot$ denotes concatenation. In both $\Sigma_A, \Sigma_A^{+}$ there is a continuous homeomorphism, respectively a $d$-to-$1$ endomorphism (called simply the shift) given as
\[
	\sigma(\underline x)=(x_{n+1})_n, 
\]
whereas in $\Sigma_A^{-}$ we use the inverse $\sigma^{-1}$.
\begin{definition}
$A$ is mixing\footnote{Equivalently, it is irreducible and aperiodic.} if there exists $M\in\Nat$ such that every entry of $A^M$ is positive.
\end{definition}
It is easy to verify that $A$ is mixing if and only if $\sigma:\Sigma_A\to\Sigma_A$ is topologically mixing.

\smallskip 

For $\underline x\in \Sigma_A$ we denote its local stable/unstable sets by 
\begin{align*}
\Wsloc{\underline x}=\{\underline{y}\in\Sigma_A: x_n=y_n, \forall n\geq 0\}\\
\Wuloc{\underline x}=\{\underline{y}\in\Sigma_A: x_n=y_n, \forall n\leq 0\}.
\end{align*}
A central remark in the theory is that one can consider $\Sigma^{+}_A$ as the space of local unstable sets of $\Sigma_A$: for $\underline{y}\in \Wuloc{\underline x}$, $\underline{y}=\underline{x}^-\cdot \underline{y}^+$.

\smallskip 
 
Given $\phi:\Sigma_A\to\Real$ and $n\geq 0$ let
\[
	\mathrm{var}_n(\phi)=\sup\{|\phi(\underline{x})-\phi({\underline{y}})|:x_k=y_k,\forall |k|\leq n\},
\]
and denote
\[
	\F_A=\{\phi: \mathrm{var}_n(\phi)\leq C\theta^n, \text{ for some }C, \theta>0\}.
\]
Similarly we can define the function spacees $\F_A^{\ast}, \ast\in\{-, +\}$. We recall the following classical lemma (see for example lemma $1.6$ in \cite{EquSta}).

\begin{lemma}\label{lem:Sinai}
If $\phi\in\F_A$ then there exists functions $\phi^{\ast},\gamma^{\ast}:\Sigma_A\to\Real,  \ast\in\{-, +\}$ such that
$\phi^{\ast}-\phi=\gamma^{\ast}-\gamma^{\ast}\circ \sigma$, and 
\begin{itemize}
	\item $\phi^{+}$ depends only on the coordinates $>0$, 
	\item $\phi^{-}$ depends only on the coordinates $\leq 0$.
	\item $\phi^{\ast}\in\F_A^{\ast}$.
\end{itemize}
\end{lemma}
It follows that we can identify $\phi^{\ast}$ as an element in $\F_A^{\ast}$. 

\subsection{Suspension} 
\label{sub:suspension}


If $R:\Sigma_A\to \Real_{>0}$ is continuous, we consider the space 
\[
	\{(\underline{x},t)\in\Sigma_A\times\Real: 0\leq t\leq R(\underline x)\}
\]
and identify 
\[
	(\underline x, R(\underline x))\sim (\sigma \underline x, 0) 
\]
to obtain a compact bundle over the circle, $S(A,R)=\Sigma_A\to \Real_{>0}/\sim$, together with the natural flow
\[
	s_t([\underline x, u])=[\underline x, u+t].
\]
\begin{definition}
The space $S(A,R)$ is the suspension of the shift map $\sigma:\Sigma_A\to \Sigma_A$ under the roof function $R$. The flow $s_t$ is the suspension flow.
\end{definition}
\begin{remark}
The sets 
\[
	D(i)=\{[\underline{x},0]:x_0=i\},\quad 1\leq i\leq d
\]
are pairwise disjoint, and $T=\bigcup_{i=1}^d D(i)$ are global transversal to the flow $s_t$, in the sense that every orbit of $s_t$ intersects $T$. 
\end{remark}

Write, using \Cref{lem:Sinai},
\[
	R^{\ast}-R=v^{\ast}-v^{\ast}\circ \sigma\quad \ast\in\{+,-\}
\]
and construct the spaces 
\begin{align*}
S^{\ast}(A,R)=\{(\underline{x},t)\in\Sigma_A\times\Real: 0\leq t\leq R^{\ast}(\underline x)\}/\sim
\end{align*}
by identifying $(\sigma^{-1}\underline{x},R^{-}(\underline x))\sim (\underline{x},0)$ in $S^{-}(A,R)$, and $(\sigma(\underline x), 0)\sim (\underline{x},R^{+}(\underline{x}))$ in $S^{+}(A,R)$. Similarly as in the case of $S(A,R)$ one obtains semi-flows in $S^{-}(A,R), S^{+}(A,R)$.

\smallskip

We can characterize easily the local unstable manifolds of a point in the suspension, when $R\in \F_A$. Define the map $\Psi :\Delta\subset S^-(A,R)\times \Sigma_A^{+}\to S(A,R)$ by
\[
	\Psi([\underline{x},t],\underline{y}) =[\underline{z},t+v^-(\underline{z})],
\]
where $\underline{z}=\underline{x}\cdot\underline{y}\in \Sigma_A$ (assuming that $A_{x_0,y_1}=1$). Then $\Psi$ is an homeomorphism and sends $\Sigma_A^{+}$ to the local unstable foliation in $S(A,R)$.

\smallskip

We finish this part recalling the following important theorem.

\begin{theorem}\label{thm:symbolic}
Given a metric Anosov flow $\mathtt f=(f_t)_t:M\to M$ and $\ep>0$, there exists $A\in\Mat{d}{\{0,1\}}$, $R\in\F_A$ strictly positive and a H\"older continuous function $\pi:S(A,R)\to M$ such that
\begin{itemize}
	\item $\pi$ is surjective, uniformly bounded to one, and $1-1$ on a residual set.
	\item $\pi\circ s_t=f_t\circ \pi, \forall t$.
	\item $\diam \pi(D(i))<\ep, \forall 1\leq i\leq d$.
\end{itemize}
If $\mathtt f$ is topologically mixing then $A$ is mixing.
\end{theorem}

The differentiable case is due to Ratner and Bowen \cite{SymbHyp,RatnerMarkov}, while the generalization to metric Anosov flows is due to Pollicott \cite{Pollicott1987}.

For a metric Anosov flow $\mathtt f$ with symbolic model $\pi:S(A,R)\to M$ as in the previous theorem, the $i$-th rectangle is 
\[
	\mathtt R_i=\pi(D(i)),
\]
and it follows that $\tilde T=\bigcup_{i=1}^d \mathtt R_i$ is a global transversal to $\mathtt f$; in the differentiable case one can even make the construction so that the relative interior of each $R_i$ is smooth.

\subsection{Ruelle-Perron-Frobenius operator and equilibrium states} 
\label{sub:ruelle_perron_frobenius_operator}

For a continuous map of a compact metric space $f:M\to M$ and a continuous function $\varphi:M\to \Real$ we denote
\begin{enumerate}
	\item $\PTM{f}{M}$ the set of $f$-invariant probability measures,
	\item $\Ptop(\varphi)=\sup_{\nu\in \PTM{f}{M}}\{h_{\nu}(f)+\int \varphi d\nu\}$ the topological pressure of the system $(f,\varphi)$,
	\item $\Eq(f,\varphi)=\{\mu\in\PTM{f}{M}:\Ptop(\varphi)=h_{\mu}(f)+\int \varphi d\mu\}$, the set of equilibrium states of $(f,\varphi)$.
\end{enumerate}
Correspondingly, for a flow $\mathtt f=(f_t)_t$ on $M$ we write
\begin{enumerate}
	\item[4.] $\PTM{\mathtt f}{M}=\bigcap_{t}\PTM{f_t}{M}$, the set of flow invariant measures,
	\item[5.] $\Ptop(\mathtt f, \varphi)= \sup_{\nu\in \PTM{\mathtt f}{M}}\{h_{\nu}(f_1)+\int \varphi d\nu\}$, the topological pressure of the system $(\mathtt f, \varphi)$,
	\item[6.] $\Eq(\mathtt f,\varphi)$ the set of equilibrium states of the system $(\mathtt f,\varphi)$.
\end{enumerate}

\begin{theorem}[Proposition $6.2$ in \cite{EqStatesCenter}]
If $\mathtt f=(f_t)_t$ is a topologically mixing Anosov flow and $\varphi$ is H\"older, then $\Eq(\mathtt f,\varphi)=\Eq(f_1,\varphi^1)=\{\muf\}$, where
\[
	\varphi^1(x)=\int_0^1 \varphi(f_t x) \der t.
\]
\end{theorem}

From now on $\mathtt f=(f_t)_t$ denotes a metric Anosov flow in the hypotheses of the \hyperlink{main}{Main Theorem} with symbolic model $\pi: S(A,R)\to M$, and $\varphi:M\to \Real$ is a fixed H\"older function. It is not loss of generality (by subtracting the pressure to $\varphi$) to assume
\[
	\Ptop(\mathtt f, \varphi)=0.
\] 

\smallskip 

We recall one possible construction of the equilibrium state associated to the system $(\mathtt f,\varphi)$. The map $\phi=\varphi\circ \pi$ is H\"older continuous, and one is led to consider equilibrium states for suspended flows. Due to the structure of these, it is enough to consider equilibrium states for $\sigma:\Sigma_A\to \Sigma_A$, replacing the potential $\phi$ by its integrated version
\[
	\tilde{\phi}(\underline{x})=\int_0^{R(\underline x)} \phi([\underline x,u])du.
\]
Using the (local) product structure of $S(A,R)$, one defines also $\psi:\Delta'\subset \Sigma^{-}_A\times S^{+}(A,R)\to\Real$ with
\[
	\psi(\underline x, [\underline y,t])=\phi(\underline x\cdot \underline{y}, t-v^{+}(\underline x\cdot \underline{y}))
\]
together with its integrated version
\[
	\tilde{\psi}(\underline{x}\cdot \underline{y})=\int_0^{R^{+}(\underline y)} \psi(\underline{x},[\underline{y},u])du.
\]
It turns out that  $\tilde{\psi}-\tilde{\phi}=k-k\circ \sigma$, with
\[
	k(\underline{x}\cdot\underline y)=\int_{-v^{+}(\underline{x}\underline y)}^{0}\phi([\underline{x}\cdot\underline y, u])du.
\]
If follows then that $\Eq(\sigma,\tilde{\phi})=\Eq(\sigma,\tilde{\psi})$.

We can now apply the technology of transfer operators: write $\tilde{\psi}^{+}=\tilde{\psi}+w-w\circ \sigma$, and for continuous valued functions $h\in \CM[\Sigma_A^{+}]$ let
\[
	\Lcal h(\underline x)=\sum_{\sigma \underline y=\underline x} e^{\tilde{\psi}^{+}(\underline y)}h(\underline{y}).
\]
Then $\Lcal \one=\one$, and the equilibrium state $\mu$ for the system $(\Sigma_A^{+},\tilde{\psi}^{+})$ is the unique eigen-measure of the adjoint of $\Lcal$. The corresponding equilibrium state for the system $(\sigma, \tilde{\phi})$ is then the (unique) shift invariant extension of $\mu$ to $\Sigma_A$, which we also denote as $\mu$.

The approach of N. Haydn (see \cite{localproductstructureHaydn}) is then defining the measure $\mu^u_{\Psi([\underline x,0],\underline{y})}$ on $\Wuloc{\Psi([\underline x,0],\underline{y})}$ as
\[
	\Psi_{\ast} e^{-(w+v)}\mu 
\]
and for a general point $\Psi([\underline x, t],\underline{y})$ with $0\leq t\leq R^{+}(\underline x)$ we define $\mu^u_{\Psi([\underline x, t],\underline y)}$, so that
\[
	(s_{-t})_{\ast}\mu^u_{\Psi([\underline x, t],\underline y)}=e^{-\int_{0}^{t}\phi([\cdot, u])\der u}\mu^u_{\Psi([\underline x, 0],\underline y)}
\]
\begin{remark}
The function $w+v$ is the transfer function between the cohomologous potentials $\tilde{\phi}, \tilde{\psi}^{+}$.
\end{remark}

\begin{theorem}
The family of measures $\mu^u=\{\mu^u_p:p\in S(A,R)\}$ satisfies:
\begin{enumerate}
 	\item each $\mu_{p}^u$ is non-atomic with full support in $\Wuloc{p}$.
 	\item It holds 
	\[
	(s_{-t})_{\ast}\mu^u_{s_t(p)}=e^{-\int_0^t \phi([\cdot,u])du}\mu^u_{p}.
	\]
    \item The family $\mu^u$ depends continuously on $p$
 \end{enumerate} 
\end{theorem}

We need to modify the previous family; denote
\begin{enumerate}
 	\item $h_{\phi}(\underline z)=e^{\tilde{\phi}(z)},$
 	\item $\Delta_{\underline x, \underline{y}}:\Wuloc{\Psi([\underline x,0],\underline{y})}\to\Real_{>0}$,
 	\[
 	\Delta_{\underline x,\underline{y}}(\Psi([\underline x,0],\underline{y}')):=\prod_{k=1}^{\oo}\frac{h_{\phi}(\sigma^{-k}\underline x\cdot y')}{h_{\phi}(\sigma^{-k}\underline x\cdot y)}.
 	\]
 	\item \(\nu_{\Psi([\underline x,0],\underline{y})}^u=\Delta_{\underline x,\underline{y}}\mu_{\Psi([\underline x,0],\underline{y})}^u\)
 \end{enumerate}
and extend for points $p=\Psi([\underline x,t],\underline{y}), 0\leq t<R^+(\underline x)$ by requiring
\[
	(s_{-t})_{\ast} \nu_{s_t(p)}^u=e^{-\int_0^t \phi([\underline x\cdot \underline y,u])du}\nu^u_{p}
\] 
By the form of these measures, there is a one to one correspondence with a family $\{\nu^u_{\underline z}:\underline{z}\in \Sigma_A\}$ satisfying the following properties.

 \begin{proposition}
 It holds:
 \begin{enumerate}
 	\item each $\nu_{\underline z}^u$ is a measure on $\Wuloc{\underline z}$.
 	\item $\sigma\nu^u_{\underline z}=e^{\tilde{\phi}(z)}\nu^u_{\sigma \underline z}$.
 	\item If $\xi=\{\Wuloc{\underline z}:\underline{z} \in \Sigma_A\}$ and $\Eq(\sigma,\phi)=\{\mm\}$, then the conditional measures of $\mm$ on $\xi$ are given as
 	\[
 	\mm^{\xi}_{\underline z}=\frac{\nu^u_{\underline z}}{\nu^u_{\underline z}(\Wuloc{\underline z})},
 	\]
 \end{enumerate}
 \end{proposition}
 \begin{proof}
 The first and the second part follow directly by construction. The third is consequence of the previous ones, and is proven (in more generality) in Section 4 of \cite{EqStatesCenter}.
 \end{proof}

Using $\pi$ we can use the measures $\nu_{p}^u$ to induce a family $\{\nu^u_x:x\in M\}$ where
each $\nu^u_x$ is a measure on a local unstable disc containing $x\in M$. This follows from the properties of the symbolic model, and is given for example in the Appendix of \cite{Ruelle1975}.

The measures defined in the Main Theorem are precisely the $\nu^u_x$. Note that these give the disintegration of the unique equilibrium state for the system $(\mathtt f, \varphi)$, and therefore they are uniquely defined modulo normalization.

\begin{remark}
Suppose that $\mm\in\PM$ for which there exists an adapted partition $\xi$ to $\Fu$ so that
its conditionals are of the form
	\[
	\mm^{\xi}_x=\frac{\nu^u_x}{\nu^u_x(\xi(x))}.
    \]
    Then the projetive class of $\nu_x^u$ can be recovered from the conditional measures, and in particular if $\xi'$ is another adapted partition to $\Fu$, then the conditional measures of $\mm$ with respect to $\xi'$ are given by 
    $\{\nu^u_x\}$.
\end{remark}

From the previous discussion it follows that to establish the \hyperlink{main}{Main Theorem} it suffices to prove:

\begin{theorem}
Let $\mm\in \PM[\Sigma_A]$ so that its disintegration along the partition $\xi$ is given by the family $\nu_{\underline z}^u$,
\[
	\mm^{\xi}_{\underline{z}}=\frac{\nu_{\underline z}^u(\cdot )}{\nu_{\underline z}^u(\Wuloc{\underline z})}.
\]
Then $\mm$ is the unique equilibrium state for the system $(\sigma, \phi)$.
\end{theorem}

We finish this part by noting that since $\sigma \Wuloc{\underline x}=\bigsqcup_{i=1}^d \Wuloc{\underline x^i}$ for some points $\underline z^i$, we can extend the $\nu_{\underline z}^u$ to measures on the whole unstable set $\Wu{\underline z}$. We will tacitly assume that in what follows.

\section{Marcus' operators} 
\label{sec:Marcus}

Given a continous function $\h:\Sigma_A\to \Real$ and $n \geq 0$ we define a new continous function 
\begin{equation}\label{eq:Rnh}
R_n\mathfrak{h}(\underline x)=\frac{1}{\nuux(\Wuloc{\underline x})}\int_{\Wuloc{\underline x}}\h\circ \sigma^n \der\nux[\underline{x}].	
\end{equation}
This way we get a family of linear contractions $\{R_n:\CM[\Sigma_A]\to\CM[\Sigma_A]:n\geq 0\}$.

\begin{proposition}\label{pro:equicontinuidad}
The family $\{R_n\h\}_n$ is equicontinuous. 
\end{proposition}
\begin{proof}
Let $\underline{x}, \underline{y}, \underline{z}\in C(a)_0$ with
\begin{itemize}
	\item $\underline{y}\in \Wuloc{\underline x}$,
	\item $\underline z\in \Wuloc{\underline{y}}$
\end{itemize}

We have $\nux=C(\underline x,\underline y)\nux[\underline y]$, and thus $R_n\h(\underline x)=R_n\h(\underline y)$. On the other hand, the difference between the functions $\h\cdot \one_{\Wuloc{\underline y}}, \h\cdot \one_{\Wuloc{\underline z}}$ converges uniformly to zero as $\underline y\mapsto \underline z$, and it follows that same is true for the quantity
\[
	\left|\frac{1}{\nux[\underline{y}](\Wuloc{\underline y})}\int_{\Wuloc{\underline y}}\h\der\nux[\underline{y}]-\frac{1}{\nux[\underline{z}](\Wuloc{\underline z})}\int_{\Wuloc{\underline z}}\h\der\nux[\underline{z}]\right|
\]
For points $\underline y, \underline{z}$ in the same stable set one can apply this fact to $\h\circ \sigma^n\cdot \one_{\Wuloc{\underline y}}, \h\circ\sigma^n \one_{\Wuloc{\underline z}}, \forall n\geq 0$.

Putting everything together and using the local product structure inside $C(a)_0$ we deduce the claim.
\end{proof}

We will now show that $\{R_n\h\}_n$ converges uniformly to some constant $c(\h)$. Let $\Omega$ be the set of all sequences $\{\Theta^m_{n,\underline x}:n,m\geq 0, \underline x\in \Sigma_A\}$ satisfying
\begin{enumerate}
\item $\Theta^m_{n,x}$ is a probability measure on $\sigma^m(\Wuloc{\underline x})$, 
\item for every $\underline x\in \Sigma_A$
\begin{equation}\label{eq:Thetameasures}
R_{n+m}\h(\underline x)=\int R_{n}\h(\underline y)\der\Theta^m_{n,\underline x}(\underline y).
\end{equation}
\end{enumerate}

\begin{definition}
$\{\Theta^m_{n,\underline x}\}\in\Omega$ is adapted to a cylinder $U$ if
\[
	\inf_{n,m, \underline x} \Theta^m_{n,\underline x}(U)>0.
\]
\end{definition}

The bulk of the work is contained in proving that for any cylinder $U$ there is $\{\Theta^m_{n,\underline x}\}\in\Omega$ adapted to it. If this were the case, note that $\{c_n(\h)=\inf_{\underline{x}}  R_{n}\h(\underline x)\}_{n\geq 0}$ is an increasing bounded sequence of real numbers, and therefore 
\[
	\exists c(\underline{h})=\lim_n c_n(\h)=\sup_n c_n(\h).
\]
Note also that if $g$ is any accumulation point of $\{R_n\h\}_n$, then $g\geq c(\underline{h})$.

Recall that $A^{m}$ has positive entries, for every $m\geq M$.

\begin{lemma}\label{lem:proporcion}
Given a cylinder $U$ there exist $C_U>0$ such that for every $\underline{x} \in \Sigma_A, \forall m>M$ it holds
\[
	\sigma^m(\Wuloc{\underline x})=\bigcup_{i=1}^{k_m} \Wuloc{\underline y^i}
\]
with
\[
	\frac{\#\{i:\underline{y}^i\in U\}}{k_m}\geq C_U.
\]
\end{lemma}
\begin{proof}
It is no loss of generality to consider the case $U=C(a_0\cdots a_l)_{0}$, since $\sigma^{k}C(a_0\cdots a_l)_{-k}=C(a_0\cdots a_l)_{0}$. Fix $\underline x$: we have
\[
\sigma^m(\Wuloc{\underline x})=\{\underline{y}: y_{n-m}=x_{n}, \forall n\leq 0  \}=\bigcup_{i=1}^{k_m} \Wuloc{\underline{z}^i}= \bigcup_{i=1}^{k_m} \Wuloc{\underline{z}^i}	
\]
where in particular for $i\neq i'$ there exists $-m<j\leq 0$ so that $z^i_j\neq z^{i'}_j$. We denote $\mathtt W=\{\underline{z}^i:1\leq i\leq  k_m\}$ and observe that we can write 
\[
	\mathtt W=\bigcup_{e=1}^d \mathtt W_e\quad \mathtt W_e=\{\underline{z}^i\in \mathtt W:z_{-M}^{i}=e\}, 
\]
where each $\mathtt W_e\neq \emptyset$. Therefore for $m\geq M$ and each $1\leq e\leq d$ the proportion
\[
	\frac{\#\{\underline{z}^i\in\mathtt W_e:z_0^i=a_0\}}{\#\mathtt W_e}
\]
is positive and independent of $m$. This implies the claim, since one can choose a fixed proportion of points $\underline y^i\in \Wuloc{\underline z^i}$ from those with $z^i_0=a_0$, also satisfying $\underline y^i\in U$ (because $U$ does not depend on $m$). 
\end{proof}

We compute, for $n,m\geq 0$ 
\begin{align*}
R_{n+m}\h(\underline{x})&=\frac{1}{\nuux(\Wuloc{\underline{x}})}\int_{\Wuloc{\underline{x}}}\h\circ \sigma^{n+m}\der\nuux\\
&=\frac{1}{\nux[\sigma^m\underline{x}](\sigma^m(\Wuloc{\underline{x}}))}\int_{\sigma^m(\Wuloc{\underline{x}}))}\h\circ \sigma^{n}\der\nu^u_{\sigma^m\underline{x}}\\
&=\frac{1}{\nux[\sigma^m\underline{x}](\sigma^m(\Wuloc{\underline{x}}))}\sum_{i=1}^{k_m}\int_{\Wuloc{\underline{y}^i}}\h\circ \sigma^n \der\nu^u_{\sigma^m\underline{x}}\\
&=\frac{1}{\nux[\sigma^m\underline{x}](\sigma^m(\Wuloc{\underline{x}}))}\sum_{i=1}^{k_m}\frac{\nux[\sigma^m \underline{x}](\Wuloc{\underline y^i})}{\nux[\underline{y}^i](\Wuloc{\underline y^i})}\int_{\Wuloc{\underline{y}^i}}\h\circ \sigma^n \der\nux[\underline{y}^i]\\
&=\sum_{i=1}^{k_m} \frac{\nux[\sigma^m \underline{x}](\Wuloc{\underline{y}^i})}{\nux[\sigma^m\underline{x}](\sigma^m(\Wuloc{\underline{x}}))} R_n\h(\underline{y}^i).
\end{align*}
It follows that $\{\Theta^m_{n,\underline x}\}\in\Omega$, where
\[
	\Theta^m_{n,\underline x}=\sum_{i=1}^{k_m} \frac{\nux[\sigma^m \underline{x}](\Wuloc{\underline{y}^i})}{\nux[\sigma^m\underline{x}](\sigma^m(\Wuloc{\underline{x}}))}\delta_{\underline{y}^i}.
\]
\begin{lemma}
$\{\Theta^m_{n,\underline x}\}$ is adapted to $U$
\end{lemma}
\begin{proof}
Indeed, by definition of the measure $\nux[\sigma^m \underline{x}]$,
\[
	\Theta^m_{n,\underline x}(U)=\frac{\sum_{\substack{i=1\\ \underline{y}^i\in U}}^{k_m} \nux[\sigma^m \underline{x}](\Wuloc{\underline{y}^i})}{\sum_{i=1}^{k_m} \nux[\sigma^m \underline{x}](\Wuloc{\underline{y}^i})}=\frac{\sum_{\substack{i=1\\ \underline{y}^i\in U}}^{k_m} \nux[\underline{y}^i](\Wuloc{\underline{y}^i})}{\sum_{i=1}^{k_m} \nux[\underline{y}^i](\Wuloc{\underline{y}^i})}.
\]
On the other hand, for every $\underline{z}, \underline{w}\in\Sigma_A$ the measures $\nux[\underline{z}](\Wuloc{\underline{z}}), \nux[\underline{w}](\Wuloc{\underline{w}})$ are uniformly comparable, therefore the result follows from \Cref{lem:proporcion}. 
\end{proof}

Consider any point of accumulation $\g\in\CM[\Sigma_A]$ of $\{R_n\h\}$, with 
\[
	\lim_{k}R_{n_k}\h=\g.
\]
As noted $g\geq c(\underline{h})$, and one can compute
\[
R_{n_k+m}\h(\underline x)-c(\underline{h})=\int (R_{n_k}\h-c(\underline{h}))\der\Theta^m_{n_k,\underline x}.
\]
We choose $\underline{x}^{n_k+m}$ with $R_{n_k+m}\h(\underline{x}^{n_k+m})=c(\underline{h})$ and consider the measures
\[
\Theta^m_{(n_k)}=\Theta^m_{n,\underline{x}^{n_k+m}}.
\]
It follows that 
\[
	0=\int (R_{n_k}\h-c(\underline{h}))\der \Theta^m_{(n_k)}
\]
hence if $\Theta$ is any accumulation point of $\{\Theta^m_{(n_k)}\}$, then
\[
	\int (\g-c(\underline{h}))\der\Theta=0.
\]
Observe that $\Theta(U)>0$, and since $\g-c(\underline{h})\geq 0$, we deduce that there exists some $\underline{x}_U\in U$ such that $g(\underline{x}_U)=c(\underline{h})$. But $U$ is arbitrary, hence $\g\equiv c(\underline{h})$. We have shown:

\begin{theorem}\label{thm:Marcus}
For every $\h \in\CM[\Sigma_A]$ the family $\{R_n\h\}_n$ converges uniformly to a constant $c(\underline{h})$.
\end{theorem}

\section{Equicontinuity of the conditional expectations}

Recall that we are denoting by $\xi$ the partition of $S(A,R)$ into local unstable manifolds,
\[
\xi(\Psi([\underline{x},t],\underline{y}))=\{\Psi([\underline{x},t],\underline{y}')):\underline{y}'\in\Sigma_A^{+}\}	
\]
and for $n\geq 0$ we denote by $\xi^n$ the partition with atoms 
\begin{align*}
\xi^n(\Psi([\underline{x},t],\underline{y}))=\{\Psi([(x_k)_{k\leq n} a_{k+1}\cdots a_0,t ],\underline y'):&\underline y'\in\Sigma_A^+,\\
&A_{x_{-k} a_{k+1}}=1, A_{a_j, a_{j+1}}=1, k+1\leq j<-1\}.
\end{align*}
It holds
\begin{equation}
\xi^n(\Psi([\underline{x},t],\underline{y}))=\bigsqcup_{j} \xi(\Psi([\underline{x}^j,t],\underline{y}))
\end{equation}
for some points $\underline{x}^j\in\Sigma_A^{-}$.

We fix $\mm$ a probability measure in $S(A,R)$ whose conditionals on the unstable sets are given by the family $\{\nu_{p}^u\}$. Given $\h\in \CM[\Sigma_A]$ the conditional expectation of $\h$ with respect to $\xi^n$ can be computed in the point $p=\Psi([\underline{x},t],\underline{y})$ as
\begin{align*}\label{eq:Enh}
E_n\h(p)=\frac{1}{\nu^u_p(\xi^n(p))}\int_{\xi^n(p)}\h \der \nu^u_p=\frac{1}{\nu^u_{\sigma^{-n}\underline z}(\Wuloc{\sigma^{-n}\underline{z}})}\int_{\Wuloc{\sigma^{-n}\underline{z}}} \h([\sigma^n\underline{w},t]) \der\nu^u_{\sigma^{-n}\underline z}(\underline{\underline w}).
\end{align*}
where $\underline{z}=\underline{x}\cdot\underline{y}$

It now follows by \Cref{thm:Marcus} that $\{E_n\h\}_n$ converges uniformly to some constant $\underline{c}(\h)$, and therefore
\[
	\int \h \der\mm=\int E_n\h \der\mm\xrightarrow[n\to\oo]{} \int c(\h)\der\mm=c(\h),
\]
which in turn implies $c(\h)=\int \h \der \mm$. This shows that there is at most one measure having conditionals given by $\{\nu_{x}^u\}$, and since the equilibrium state for the potential $\phi$ satisfies this condition, it is this referred measure. This concludes the proof of the \hyperlink{main}{Main Theorem}.

\printbibliography

\end{document}